\documentclass[english,10pt]{article}
\usepackage{graphicx}
\usepackage{amsmath}
\usepackage{amssymb}
\usepackage{pst-node}
\usepackage{color}
\usepackage{graphicx}
\usepackage{epstopdf}
\usepackage{tikz}
\usepackage{pgfplots}
\usepackage{pgflibraryarrows}
\usepackage{pgflibrarysnakes}
\usepackage{ifpdf}

\newtheorem{theorem}{Theorem}[section]
\newtheorem{corollary}[theorem]{Corollary}
\newtheorem{definition}[theorem]{Definition}
\newtheorem{conjecture}[theorem]{Conjecture}
\newtheorem{observation}[theorem]{Observation}

\newtheorem{lemma}[theorem]{Lemma}

\newenvironment{proof}{\noindent {\bf Proof.}}{\rule{3mm}{3mm}\par\medskip}

\newcommand{\Z}{\mbox{$\mathbb Z$}}

\newcommand{\setCZ}{\mbox{${\mathcal S}_3$}}

\begin{document}
\title{Complementary Graphs with Flows Less Than Three}

\author{
\small  Jiaao Li$^1$, Xueliang Li$^2$, Meiling Wang$^2$\\
\small $^1$School of Mathematical Sciences and LPMC \\
\small Nankai University, Tianjin 300071, China \\
\small $^2$Center for Combinatorics and LPMC \\
\small Nankai University, Tianjin 300071, China \\
\small Emails: lijiaao@nankai.edu.cn; lxl@nankai.edu.cn; Estellewml@gmail.com
}

 \date{}
\maketitle
\begin{abstract}
X. Hou, H.-J. Lai, P. Li and C.-Q. Zhang [J. Graph Theory 69 (2012) 464-470] showed  that for a simple graph $G$ with $|V(G)|\ge 44$, if $\min\{\delta(G),\delta(G^c)\}\ge 4$,
then either $G$ or its complementary graph $G^c$ has a nowhere-zero $3$-flow. In this paper, we improve this result by showing that if $|V(G)|\ge 32$ and $\min\{\delta(G),\delta(G^c)\}\ge 4$, then either $G$ or $G^c$ has flow index strictly less than $3$. Our result is proved by a newly developed closure operation and contraction method.
\\[2mm]
\\[2mm] \textbf{AMS Subject Classification (2010):} 05C21, 05C40, 05C07
\end{abstract}

\section{Introduction}

Graphs in this paper may contain parallel edges but no loops. We call a graph {\em simple} if it contains no parallel edges. An integer flow of a graph $G$ is an ordered pair $(D, f)$, where $D$ is an orientation of $G$ and $f$ is a mapping from $E(G)$ to the set of integers such that the incoming netflow equals the outgoing netflow at every vertex. A flow $(D, f)$ is called a {\em nowhere-zero $k$-flow} if $f(e)\in\{\pm 1, \pm 2, \cdots, \pm(k-1)\}$ for every edge $e\in E(G)$. Tutte proposed several celebrated flow conjectures, and the $3$-flow conjecture is stated as follows.
\begin{conjecture}[Tutte's $3$-Flow Conjecture, 1972]
Every $4$-edge-connected graph has a nowhere-zero $3$-flow.
\end{conjecture}

Jaeger \cite{Jaeger79} in  1979 showed that every $4$-edge-connected graph has a nowhere-zero $4$-flow. In 2012, Thomassen \cite{Thom12} made a breakthrough on this conjecture  by showing that every $8$-edge-connected graph has a nowhere-zero $3$-flow. This was later improved by Lov\'asz, Thomassen, Wu and Zhang \cite{LTWZ13}.
\begin{theorem}{\em (Lov\'asz et al. \cite{LTWZ13})}\label{LTWZ2013thm}
Every $6$-edge-connected graph has a nowhere-zero $3$-flow.
\end{theorem}

Besides the edge connectivity conditions, Hou, Lai, Li and Zhang \cite{HLLZ12} studied the $3$-flow property of a graph $G$ and its complementary graph $G^c$, providing another evidence to Tutte's $3$-flow conjecture.
\begin{theorem}\label{Hou12} {\em (Hou et al. \cite{HLLZ12})}
Let $G$ be a simple graph with $|V(G)|\ge 44$. If $\min\{\delta(G),\delta(G^c)\}\ge 4$,
then either $G$ or $G^c$ has a nowhere-zero $3$-flow.
\end{theorem}

For integers $k\ge 2d>0$, a {\em circular $k/d$-flow} is an integer flow $(D, f)$ such that $f$ takes values from $\{\pm d, \pm(d+1), \dots, \pm(k-d)\}$. When $d=1$, this is exactly the nowhere-zero $k$-flow.
The {\em flow index $\phi(G)$} of a graph $G$ is the least rational number $r$ such that $G$ admits a circular
$r$-flow. It was proved in \cite{GTZ98} that such an index indeed exists, and the circular flow satisfies the monotonicity that for any pair of
rational numbers $r\ge s$, a graph admitting a circular $s$-flow has a circular $r$-flow as well. Thus circular flows are refinements of integer flows.

A {\em modulo $3$-orientation} is an orientation $D$ of $G$ such that the outdegree is congruent to the indegree modulo $3$ at each vertex. It is well-known that a graph admits a nowhere-zero $3$-flow if and only if it admits a modulo $3$-orientation (see \cite{Jaeger88,Younger83,Zhang97}).
The study of flow index strictly less than $3$ is initiated in \cite{LTWZ18} with the following theorems.

\begin{theorem}\label{SClemma}{\em(\cite{LTWZ18})}
A graph $G$ satisfies $\phi(G) < 3$ if and only if
$G$ has a strongly-connected modulo $3$-orientation.
\end{theorem}

\begin{theorem}{\em(\cite{LTWZ18})}
\label{<3}
For every $8$-edge-connected graph $G$, the flow index $\phi(G)<3$.
\end{theorem}

It is worth noting that (see \cite{Jaeger88}) if $\phi(G)\le 5/2$ for every $9$-edge-connected graph $G$, then Tutte's $5$-Flow Conjecture follows, that is,  $\phi(G)\le 5$ for every bridgeless graph $G$.
Since $K_6$ has only one modulo 3-orientation up to isomorphism which is not strongly-connected, we have $\phi(K_6)=3$, and so Theorem \ref{<3} cannot be extended to $5$-edge-connected graphs.
In \cite{LTWZ18}, it was conjectured that the $6$-edge-connectivity suffices for $\phi(G)<3$.
\begin{conjecture}\label{CONJ:<3}{\em(\cite{LTWZ18})}
For every $6$-edge-connected graph $G$, the flow index $\phi(G)<3$.
\end{conjecture}

In this paper, we aim to extend Theorem \ref{Hou12} in the theme of flow index $\phi<3$. Our main result is as follows, providing further evidence to Conjecture \ref{CONJ:<3}.

\begin{theorem}\label{thmmain1}
Let $G$ be a simple graph with $|V(G)|\ge 32$. If $\min\{\delta(G),\delta(G^c)\}\ge 4$,
then\\ $\min\{\phi(G), \phi(G^c)\}<3$.
\end{theorem}

Theorems \ref{LTWZ2013thm} and \ref{Hou12} were proved by using the {\em group connectivity} ideas, which allows flow with boundaries.  Let $G$ be a graph, and let $Z(G,\Z_{3}) =\{\beta: V(G)\rightarrow \Z_{3}\;|$
$\sum_{v\in V(G)}\beta(v) \equiv 0  \pmod {3}\}$. Given a boundary function $\beta\in Z(G,\Z_{3})$, an orientation $D$ of $G$ is called a {\em $\beta$-orientation} if $d^+_D(v)-d^-_D(v)\equiv \beta(v)  \pmod {3}$ for every vertex $v \in V (G)$.
A graph $G$ is {\em  $\Z_{3}$-connected} if $G$ has a $\beta$-orientation for every
$\beta\in Z(G,\Z_{3})$. It follows from the definition that every $\Z_{3}$-connected graph admits a modulo $3$-orientation and hence has a nowhere-zero $3$-flow.
In fact, Hou et al. \cite{HLLZ12} obtained a stronger version of Theorem \ref{Hou12} on  $\Z_3$-group connectivity.
\begin{theorem} \label{Hou2} {\em (Hou et al. \cite{HLLZ12})}
Let $G$ be a simple graph with $|V(G)|\ge 44$. If $\min\{\delta(G),\delta(G^c)\}\ge 4$,
then either $G$ or $G^c$ is $\Z_{3}$-connected.
\end{theorem}

Motivated by Theorem \ref{SClemma}, we develop a contractible configuration method to handle the flow index $\phi<3$ problem in this paper, which is analogous to the $\Z_3$-group connectivity.

\begin{definition}\label{DEF:CZ3}
A graph $G$ is {\bf strongly-connected $\Z_{3}$-contractible} if, for every
$\beta\in Z(G,\Z_3)$, there is a strongly-connected orientation $D$ such
that
$d^+_D(v)-d^-_D(v)\equiv \beta(v)  \pmod {3}$ for every vertex $v \in V (G)$. Let  $\setCZ$
denote the family of all strongly-connected $\Z_{3}$-contractible graphs.
\end{definition}

A strongly-connected $\Z_{3}$-contractible graph is called a {\bf $\setCZ$-graph} for convenience.  A $\setCZ$-graph is $\Z_3$-connected by definition; and it has flow index less than $3$ by Theorem \ref{SClemma}. Actually, it was  proved in Theorem 4.2 of \cite{LTWZ18} that $G\in\setCZ$ for every $8$-edge-connected graph $G$.

In this paper, we shall prove a $\setCZ$ version of Theorem \ref{thmmain1}. However, a directed $\setCZ$-property like Theorem \ref{Hou2} fails, and there are some exceptions. A {\bf bad attachment} of a graph $G$ is an induced subgraph $\Gamma$ with $3\le |V(\Gamma)|\le 6$ and there are at most $3|V(\Gamma)|-|E(\Gamma)|$ edges between $V(\Gamma)$ and $V(G)\setminus V(\Gamma)$ in $G$. We will see later (Remark 1 in Section 2) that if a graph $G$ contains a bad attachment, then $G\notin\setCZ$.
We obtain the $\setCZ$ version of Theorem \ref{thmmain1} as follows.
\begin{theorem}\label{main2m}
Let $G$ be a simple graph with $|V(G)|\ge 73$. If $\min\{\delta(G),\delta(G^c)\}\ge 4$,
then one of the following holds:
\begin{itemize}
  \item [(i)] $G\in \setCZ$ or $G^c\in\setCZ$,
  \item [(ii)] both $G$ and $G^c$ contains a bad attachment.
\end{itemize}
Moreover, in case (ii) we have both $\phi(G)<3$ and $\phi(G^c)<3$.
\end{theorem}
In fact, if case (ii) of Theorem \ref{main2m} occurs, we obtain a more detailed characterization of bad attachment in Theorem \ref{main2s} of Section $3$. Also, the graph obtained by deleting the bad attachment is a special kind of contractible graph for $\phi<3$ property to be introduced in Section 2.  Furthermore, if we impose the minimal degree condition to $\min\{\delta(G),\delta(G^c)\}\ge 5$, then an easy counting argument shows that case (ii) of Theorem \ref{main2m} cannot happen. (See Theorem \ref{main2s} below for more details.) Thus we have the following corollary.
\begin{corollary}
Let  $G$ be a simple graph with $|V(G)|\ge 73$. If $\min\{\delta(G),\delta(G^c)\}\ge 5$,
then $G\in \setCZ$ or $G^c\in\setCZ$.
\end{corollary}

In the next section, we will present some preliminaries. The proofs of Theorems \ref{thmmain1} and \ref{main2m} will be given in Section 3.  We end this section with a few more notation.

{\bf Notation.}
 A vertex of degree at least $k$ is called a {\em $k^+$-vertex}. Let $X$, $Y$ be disjoint subsets of vertices of graph $G$. We denote the set of edges between $X$ and $Y$ in $G$ by $E_G(X,Y)$, and let $e_G(X,Y)=|E_G(X,Y)|$. When $X=\{x\}$ or $Y=\{y\}$, we use $E_G(x,Y)$, $E_G(X,y)$, $E_G(x,y)$ and $E_G(u)=E_G(\{u\},V(G)\setminus\{u\})$  for short.
For a vertex set $A\subseteq V(G)$, we denote by $G/A$ the graph obtained from $G$ by identifying the vertices of $A$ into a single vertex and deleting the resulting loops. Moreover, we use $G/H$ for $G/V(H)$ when $H$ is a connected subgraph of $G$.

\section{Preliminaries}

The following observation comes straightly from Definition \ref{DEF:CZ3} of $\setCZ$-graph. This indicates that the $\setCZ$-property is closed under contraction and adding edges. It would also be useful to determine that some graphs are not in $\setCZ$.
\begin{observation}\label{inandnotin}
Let $x, y$ be two vertices of $G$. If $G\in \setCZ$, then $G+xy\in \setCZ$ and $G/\{x,y\}\in \setCZ$. Conversely, if there is a subset $X\subsetneq V(G)$ of vertices such that $G/X\notin\setCZ$, then $G\notin\setCZ$.
\end{observation}

\subsection{Contractible configurations and $3$-closure operations}

\begin{lemma}
\label{reduc-lem}
Let $G$ be a connected graph with $\beta\in Z(G,\Z_3)$, and $H$ a subgraph of $G$ and $G'=G/H$. Define a boundary function
$\beta'$ of $G'$ as follows.
\begin{equation*}\beta'(v)=
    \left\{
    \begin{array}{ll}
    \beta (v), & \text{if}~v\in V(G/H)\setminus\{v_H\},\\
    \sum\limits_{x\in V(H)}\beta(x), &\text{if}~v=v_H,
    \end{array}
    \right.
\end{equation*}
where $v_H$ denotes the vertex by contracting $H$ in $G'$. Then $\beta'\in Z(G',\Z_3)$.

If $H\in \setCZ$, then every strongly-connected $\beta'$-orientation of $G'$ can be extended to a strongly-connected $\beta$-orientation of
$G$. In particular, each of the following statements holds.
\begin{enumerate}
\item[(i)] If $H\in \setCZ$ and $\phi(G/H)<3$, then $\phi(G)<3$.
\item[(ii)] If $H\in \setCZ$ and $G/H\in \setCZ$, then $G\in \setCZ$.
\end{enumerate}
\end{lemma}
\begin{proof} Since $\sum_{x\in V(G')}\beta'(x)=\sum_{x\in V(G)\setminus V(H)}+\sum_{x\in V(H)}\beta(x)\equiv 0\pmod3$, we have $\beta'\in Z(G',\Z_3)$.
For a strongly-connected $\beta'$-orientation $D'$ of $G'$, it results a $\beta_1$-orientation $D_1$ of $G-E(H)$ (we may arbitrarily orient the edges in $E(G[V(H)])\setminus E(H)$ here). Define a function $\beta_2: V(H)\mapsto \Z_3$ by $\beta_2(v)=\beta(v)-\beta_1(v)$ for each $v\in V(H)$. Then $\sum_{v\in V(H)}\beta_2(x)=\sum_{v\in V(H)}\beta(v)-\sum_{v\in V(H)}\beta_1(v)=\beta'(v_H)-(d_{D'}^+(v_H)-d_{D'}^-(v_H))\equiv 0\pmod3$, and so $\beta_2\in Z(H,\Z_3)$. Since $H\in \setCZ$, there is a strongly-connected $\beta_2$-orientation $D_2$ of $H$. Now $D_1\cup D_2$ is a $\beta$-orientation of $G$. Since both $D_2$ and $D'=(D_1\cup D_2)/D_2$ are strongly-connected, $D_1\cup D_2$ is strongly-connected.

(i) If $H\in \setCZ$, then a strongly-connected modulo $3$-orientation of $G/H$ can be extended to $G$. Hence (i) follows from Theorem \ref{SClemma}.

(ii) Since $\beta$ is arbitrary, $G\in \setCZ$ by definition.
\end{proof}

Since a graph with $3$-edge-cuts cannot have a strongly-connected modulo $3$-orientation, it has flow index at least $3$ by Theorem \ref{SClemma}. So our study of flow index $\phi<3$ only focuses on $4$-edge-connected graphs. A graph is called {\em ($\phi<3$)-contractible}  if for every $4$-edge-connected supergraph $G$ containing $H$ as a subgraph, $\phi(G)<3$  if and only if $\phi(G/H)<3$. Clearly, a $\setCZ$-graph is ($\phi<3$)-contractible by (i) of Lemma \ref{reduc-lem}. We will show below that a wider class of graphs is also ($\phi<3$)-contractible.

\begin{lemma}\label{dipath}
{\em(\cite{LTWZ18})}
Let $G$ be a $2$-edge-connected graph, and $e=xy$ an edge of $G$. If $G/e$ has a strongly-connected orientation $D'$, then $D'$ can be extended to a strongly-connected orientation $D$ of $G$.
\end{lemma}

\begin{lemma}\label{3edge}
\label{parallelledges}
Let $G$ be a $4$-edge-connected graph with $\beta\in Z(G,\Z_3)$ and $x, y$ be a pair of vertices joined by a set $E(x,y)$ of at least $3$ parallel edges.
Let $G' = G/E(x,y)$ and $\beta'$ be the resulting $\Z_3$ boundary function, where $\beta'(v)=\beta(v)$ for any $v\in V(G)\setminus\{x,y\}$, and
$\beta'(w) \equiv \beta(x)+\beta(y)\pmod 3$ for the contracted vertex $w$.
If $G'$ has a strongly-connected $\beta'$-orientation $D'$, then $D'$ can be extended to a strongly-connected $\beta$-orientation $D$ of $G$.
\end{lemma}
\begin{proof} Let $e_1, e_2$ be two distinct parallel edges in $E(x,y)$. Then $G-e_1-e_2$ is 2-edge-connected since $G$ is 4-edge-connected, and hence we can extend $D'$ to a strongly-connected orientation of $G-e_1-e_2$ by Lemma \ref{dipath}. Note that two parallel edges $e_1, e_2$ are enough to modify the boundaries of the end vertices $x, y$. Now we appropriately orient $e_1, e_2$ to modify the boundary $\beta(x), \beta(y)$. This results a strongly-connected $\beta$-orientation $D$ of $G$.
\end{proof}

In particular, Lemma \ref{parallelledges} indicates that the graph formed by three or more parallel edges is ($\phi<3$)-contractible.

\begin{definition}\label{DEF:cl3}
Let $H$ be a subgraph of $G$. The {\bf $3$-closure} of $H$ in $G$, denoted by $cl_3(H)$, is the unique maximal induced subgraph of $G$ that contains $H$ such that $V(cl_3(H))\setminus V(H)$ can be ordered as a sequence $\{v_1, v_2, \ldots, v_t\}$ such that $e_G(v_1, V(H))\ge 3$ and for each $i$ with $1\le i\le t-1$,
$$e_G(v_{i+1},V(H)\cup\{v_1, v_2, \ldots, v_i\})\ge 3.$$
\end{definition}

Notice that for each vertex $v\in V(G)\setminus V(cl_3(H))$, we have $e_G(v,cl_3(H))\leq 2$ by the definition. The following lemma tells that if $H\in \setCZ$, then $cl_3(H)$ is also ($\phi<3$)-contractible.

\begin{lemma}\label{3closure}
Let $G$ be a $4$-edge-connected graph with a subgraph $H$. Then each of the following statements holds.
\begin{enumerate}
\item[(i)] If $H\in \setCZ$ and $\phi(G/cl_3(H))<3$, then $\phi(G)<3$.
\item[(ii)] If $H\in \setCZ$ and $G/cl_3(H)\in \setCZ$, then $G\in \setCZ$.
\end{enumerate}
\end{lemma}

\begin{proof}
 (i) Let $\{v_1, v_2, \ldots, v_t\}$ be the ordered sequence of $V(cl_3(H))\setminus V(H)$ as in Definition \ref{DEF:cl3}. Denote $H_i=G[V(H)\cup \{v_1, v_2, \ldots, v_{t+1-i}\}]$ for each $1\le i\le t$ and $H_{t+1}=H$. By Lemma \ref{3edge}, we first extend a strongly-connected modulo $3$-orientation of $G/cl_3(H)=G/H_1$ to $G/H_2$.  By applying  Lemma \ref{3edge} recursively, we can extend a strongly-connected modulo $3$-orientation of $G/H_{i}$ to $G/H_{i+1}$ for each $i=1,2,\ldots, t$. Then we apply Lemma \ref{reduc-lem} to extend this strongly-connected modulo $3$-orientation of $G/H$ to a strongly-connected modulo $3$-orientation of $G$.

 (ii) The proof of (ii) is similar to that of (i) with strongly-connected $\beta$-orientation replacing strongly-connected modulo $3$-orientation.
\end{proof}

\subsection{Properties of contractible graphs }

By Theorem $4.2$ of \cite{LTWZ18}, we have the following theorem.
\begin{theorem}\label{8edge}{\em (\cite{LTWZ18})}
For every $8$-edge-connected graph $G$,  $G\in \setCZ$.
\end{theorem}

A graph is called {\em trivial} if it is a singleton $K_1$, and {\em nontrivial} otherwise. The following lemma is due to Nash-Williams \cite{Nash64} in terms of matroids, and a detailed proof can be found in Theorem 2.4 of \cite{YLL10}.
\begin{lemma}\label{nash}
(Nash-Williams \cite{Nash64}) Let $G$ be a nontrivial graph and let $k > 0$ be an integer. If $|E(G)| \geq k(|V(G)|-1)$, then $G$ has a nontrivial subgraph $H$ such that $H$ contains $k$ edge-disjoint spanning trees.
\end{lemma}

Theorem \ref{8edge} and Lemma \ref{nash} immediately imply the following lemma, which shows that graphs with enough edges must have a nontrivial \setCZ-subgraph.
\begin{lemma}\label{czsize16}
Let $G$ be a simple graph with $|E(G)| \geq 8(|V(G)|-1)$. Then $G$ has a nontrival subgraph $H\in \setCZ$ with $|V(H)|\ge 16$.
\end{lemma}
\begin{proof}
By Lemma \ref{nash}, $G$ has a nontrivial subgraph $H$ that contains $8$ edge-disjoint spanning trees. Clearly, $H$ is $8$-edge-connected, and so $H\in \setCZ$ by Theorem \ref{8edge}. If $H$ is a simple graph, then $|V(H)|\ge 16$ follows from that $H$ contains $8$ edge-disjoint spanning trees.
\end{proof}

On the other hand, we also show that an $\setCZ$-graph cannot be too sparse.
\begin{lemma}\label{totaledge}
If a nontrivial graph $G$ belongs to $\setCZ$, then $|E(G)|\geq3|V(G)|-2$.
\end{lemma}
\begin{proof}
Fix a vertex $x\in V(G)$, define a boundary function $\beta:V(G)\rightarrow \Z_3$ by
\begin{equation*}\beta(v)\equiv
\left\{
         \begin{array}{lr}
         \sum\limits_{ y\in V(G)\setminus\{x\}}d_G(y) \pmod3, & \text{if}~v=x,\\
          -d_G(v)~~~~~~~~~~~~\pmod3, & \text{if}~v\neq x.
         \end{array}.
\right.
\end{equation*}
Clearly, $\sum_{v\in V(G)}\beta(v)\equiv 0\pmod3$ and $\beta\in Z(G,\Z_3)$.
Since $G\in \setCZ$, there is a strongly-connected $\beta$-orientation $D$ of $G$, that is, $\beta(v)\equiv d_D^+(v)-d_D^-(v)=2d_D^+(v)-d_G(v)\pmod3$ for any vertex $v\in V(G)$. For any vertex $v\in V(G)\setminus\{x\}$, since $\beta(v)\equiv -d_G(v)\pmod3$, we have $d_D^+(v)\equiv0\pmod3$, and so $d_D^+(v)\ge 3$ as a positive integer since $D$ is strongly-connected. Moreover, $d_D^+(x)\geq 1$ since $D$ is strongly-connected. Therefore, $$|E(G)|=\sum_{v\in V(G)} d_D^+(v)=d_D^+(x)+\sum_{v\in V(G)\setminus\{x\}} d_D^+(v)\geq 1+3(|V(G)|-1)=3|V(G)|-2.$$
\end{proof}
\noindent{\bf Remark 1:}  If a graph $G$ contains a bad attachment $\Gamma$, then for $X=V(G)\setminus V(\Gamma)$, the graph $G/X$ has $|V(\Gamma)|+1$ vertices and at most $3|V(\Gamma)|$ edges. Thus $G/X\notin \setCZ$ by Lemma \ref{totaledge}, and so $G\notin \setCZ$ by Observation \ref{inandnotin}.\\

Now we develop some techniques to find  $\setCZ$-graphs from smaller graphs.
For a graph $G$ with a $4^+$-vertex $v$ and $va, vb\in E_G(v)$, define $G_{[v,ab]}=G-v+ab$ as the graph obtained from $G$ by deleting the vertex $v$ and adding a new edge $ab$.

\begin{lemma}\label{split1}
Let $v$ be a $4^+$-vertex of a graph $G$ with $va, vb\in E_G(v)$. If $G_{[v,ab]}\in\setCZ$, then $G\in\setCZ$.
\end{lemma}
\begin{proof}
Let $\beta\in Z(G,\Z_3)$. We first orient all the edges of $E_G(v)\setminus\{va,vb\}$ to modify the boundary $\beta(v)$. Note that this is possible since $|E_G(v)\setminus\{va,vb\}|\ge 2$.  Then delete the oriented edges and change the boundaries of the end vertices other than $v$. Specifically, for each edge $vx \in E_G(v)\setminus\{va,vb\}$ that we oriented, increase or decrease the boundary function of $x$ by $1$ depending on the orientation of $vx$ is into $x$ or out of $x$. This results a boundary function $\beta'$ of $G_{[v,ab]}$. Since $G_{[v,ab]}\in\setCZ$, there exists a strongly-connected $\beta'$-orientation $D'$ of $G_{[v,ab]}$. By adding those deleted oriented edges and replacing the edge $ab$ by $av, vb$ (and keep their orientation), we obtain a strongly-connected $\beta$-orientation of $G$. This argument holds for any $\beta\in Z(G,\Z_3)$, and hence $G\in\setCZ$.
\end{proof}

\begin{lemma}\label{split2}
Let $G$ be a $4$-edge-connected graph and $u,v$ be two adjacent vertices in $G$. Assume that $e_G(v, V(G)\setminus\{u,v\})\ge 3$ and let $va, vb\in E(v,V(G)\setminus\{u,v\})$. Denote $G_1=G-u-v+ab$. If $G_1\in \setCZ$, then $G\in \setCZ$.
\end{lemma}
\begin{proof}
If $u$ has just one neighbor $v$, then there are at least $4$ parallel edges between $uv$. By Lemmas \ref{3edge} and \ref{3closure}, if $G/uv=G_1 \in\setCZ$, then $G\in \setCZ$.

So we assume that $u$  has at least two neighbors. Let $c\neq v$ be a neighbor of $u$, and $H=G-u+vc$. Then $H_{[v,ab]}=G-u-v+ab=G_1\in \setCZ$. Since $e_G(v, V(G)\setminus\{u,v\})\ge 3$, we know that $v$ is a $4^+$-vertex of  $H$, and so $H\in \setCZ$ by Lemma \ref{split1}. Notice that $u$ is a $4^+$-vertex of $G$ and $H=G_{[u,vc]}\in \setCZ$. Hence $G\in \setCZ$ by Lemma \ref{split1} again.
\end{proof}

\noindent{\bf Remark 2:} The condition ``$e_G(v, V(G)\setminus\{u,v\})\ge 3$'' in Lemma \ref{split2} cannot be dropped. If there are exactly two  parallel edges between $u$ and $v$ in $G$ and both $u$ and $v$ have exactly two other edges connecting $V(G)\setminus\{u,v\}$, then this graph $G$ does not belong to $\setCZ$ by Observation \ref{inandnotin}.

\subsection{Special contractible graphs}

Let $mK_2$ be the graph with two vertices and $m$ parallel edges. Let $K_3^1$, $K_3^2$, and $K_4^*$ be the graphs as depicted in Figure \ref{Hn4m12}.

\begin{figure}[!hpbt]
    \centering
    \includegraphics[width=1\textwidth,trim=65 600 65 80,clip]{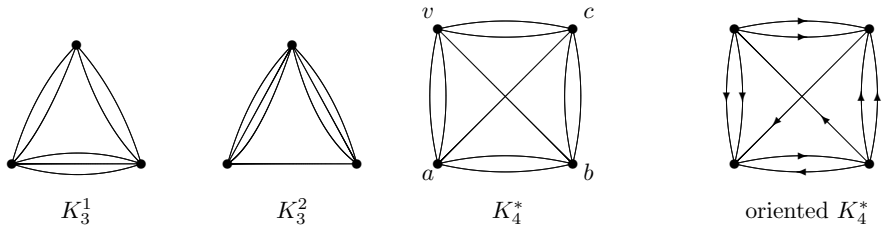}
    \caption{\small\it The graphs $K_3^1$,$K_3^2$,$K_4^*$ and strongly-connected mod $3$-orientation of $K_4^*$.}\label{Hn4m12}
\end{figure}

\begin{lemma}\label{H}
(i) $mK_2\in\setCZ$ if and only if $m\ge4$.

(ii) $K_3^1, K_3^2, K_4^*\in\setCZ$.
\end{lemma}
\begin{proof}
(i) By Lemma \ref{totaledge}, we have that $mK_2\in\setCZ$ implies $m\ge4$. When $m\ge 4$, we first orient two of the edges in the opposite direction to obtain a digon. Then there are at least two edges remaining, and we can use them to modify the boundaries of end vertices. This gives a strongly-connected $\beta$-orientation for any given boundary function $\beta$, and so $mK_2\in\setCZ$.

(ii) For $K_3^1, K_3^2$, each of them contains a $3K_2$, and contracting a $3K_2$ results a $4K_2\in\setCZ$. So $K_3^1, K_3^2\in\setCZ$ by Lemma \ref{3edge}.

Let $\beta\in Z(K_4^*,\Z_3)$. If $\beta=0$ at each vertex, then a strongly-connected modulo $3$-orientation of $K_4^*$ is in the last graph of Figure \ref{Hn4m12}. Otherwise, without loss of generality, we may assume $\beta(v)=\alpha \in\{-1,1\}$.
Consider a graph $G_1=K_4^*-v+ab+ac$ with boundary $\beta_1$ such that $\beta_1(a)=\beta(a)$, $\beta_1(b)=\beta(b)$ and $\beta_1(c)=\beta(c)+\alpha$. Then $\beta_1\in Z(G_1,\Z_3)$ and $G_1\cong K_3^1\in\setCZ$, and there exists a strongly-connected $\beta_1$-orientation of $G_1$. In $K_4^*$, replace the added edges $ab, ac$ by $av, vb$ and $av, vc$ with their orientation preserved, respectively. Then orient the remaining edge $vc$ of $K_4^*$ from $v$ to $c$ if $\alpha=1$, and from $c$ to $v$ if $\alpha=-1$. This gives a strongly-connected $\beta$-orientation of $K_4^*$. Hence $K_4^*\in\setCZ$.
\end{proof}

Now we show that some complete bipartite graphs are in $\setCZ$. Note that $K_{4,9}$ has $13$ vertices and $36$ edges, and so $K_{4,9}\notin\setCZ$ by Lemma \ref{totaledge}.
\begin{lemma}\label{bi}
When $m\geq 4$ and $n\geq 10$, we have $K_{m,n}\in\setCZ$.
\end{lemma}
\begin{proof}
We first show $K_{4,10}\in\setCZ$. Let $(X,Y)$ be a bipartition of $K_{4,10}$ with $X=\{x_1, x_2, x_3, x_4\}$ and  $Y=\{y_i|1\leq i \leq10\}$.
We apply Lemma \ref{split1} to delete vertices in $Y$ and add edges in $X$. For $1\le i\le 4$, we delete $y_{2i-1},y_{2i}$ and add two parallel edges $x_ix_{i+1}$, where $x_5=x_1$. Then delete $y_9,y_{10}$ and add edges $x_1x_3,x_2x_4$. Now the remaining graph is isomorphic to $K_4^*\in\setCZ$. By applying Lemma \ref{split1} recursively, we conclude that $K_{4,10}\in\setCZ$.

When $m\geq 4$ and $n\geq 10$, $K_{m,n}$ is $4$-edge-connected. Pick a subgraph $K_{4,10}$ in $K_{m,n}$.  Then it is easy to see that $K_{m,n}=cl_3(K_{4,10})$. Since $K_{4,10}\in\setCZ$, we have $K_{m,n}\in\setCZ$ by Lemma \ref{3closure}(ii).
\end{proof}

By Observation \ref{inandnotin}, if a graph $G$ contains $K_{m,n}\in\setCZ$ as a spanning subgraph with $m\geq 4$ and $n\geq 10$, then $G\in\setCZ$. We shall prove a similar proposition below when $G$ contains $K_{3,t}$  as a spanning subgraph and $t$ is large ($t\ge 14$ suffices).

For an integer $t\ge 4$, a $4$-edge-connected graph on $t+3$ vertices is denoted by $K_{3,t}^+$ if it contains $K_{3,t}$  as a spanning subgraph.

\begin{lemma}\label{3t}
For $t\geq 14$, $K_{3,t}^+\in\setCZ$.
\end{lemma}
\begin{proof}
Let $(A,B)$ be a bipartition of $G=K_{3,t}^+$ with $|A|= t,|B|=3$ and $E(A, B)$ contains a complete bipartite graph $K_{3,t}$. Denote $B=\{x,y,z\}$. Our strategy is to apply Lemmas \ref{split1} and \ref{split2} to delete vertices in $A$ and add edges to $B$ such that the part of $B$ forms a graph $K_3^1\in\setCZ$. Note that in the part of $B$, we need to add at most $7$ edges to form a $K_3^1$. In the part of $A$ , we can delete a vertex or two adjacent vertices and add any one of $xy, xz, yz$ by using Lemmas \ref{split1} and \ref{split2}. We will proceed to add two parallel edges $xy$, two parallel edges $xz$ and three parallel edges $yz$. The only concern is that we need to keep the remaining graph $4$-edge-connected.

Let $C_1, C_2, \ldots, C_s$ be all the components of $G[A]$. Given a component $C_i$ where $1\le i\le s$. We first note that the operations of the following cases keep the remaining graph $4$-edge-connected. If $|V(C_i)|=1$, then it means that there are parallel edges between $V(C_i)$ and some vertex of $B$, and we can delete the vertex $V(C_i)$ and add a new edge in $B$ by using Lemma \ref{split1}. If $|V(C_i)|=2$, then  there are two adjacent vertices $u, v$ in $V(C_i)$. Clearly, $e_G(v, B)\ge 3$ and Lemma \ref{split2} is applied. In this case we delete $V(C_i)$ and add a new edge in $B$. If $|V(C_i)|\ge 3$, we pick a spanning tree of $C_i$, and then delete a pendent vertex in the tree and add a new edge in $B$ by using Lemma \ref{split1} iteratively,  until this component becomes two adjacent vertices. Now we use Lemma \ref{split2} to delete this last two vertices and add a new edge in $B$. In total, all those operations could add at least
\[
\sum\limits_{|V(C_i)|\ge 2}(|V(C_i)|-1) + \sum\limits_{|V(C_i)|=1}|V(C_i)|\ge
\sum\limits_{|V(C_i)|\ge 2}\frac{|V(C_i)|}{2}+
\sum\limits_{|V(C_i)|=1}|V(C_i)|\ge \frac{1}{2}\sum_{i=1}^{s}|V(C_i)|\ge7
 \]
 edges to part $B$.

Therefore, we can successfully apply these operations to obtain a $K_3^1\in\setCZ$ in part $B$, and the resulting graph is $4$-edge-connected and it is formed by $cl_3(K_3^1)$. Hence it is in $\setCZ$ by  (ii) of Lemma \ref{3closure}. By using Lemmas \ref{split1} and \ref{split2} recursively, we can get that $K_{3,t}^+\in\setCZ$.
\end{proof}

As mentioned in the introduction, we have $\phi(K_6)=3$; and there is another $5$-edge-connected planar graph $2C_5\cdot K_1$ on $6$ vertices with flow index exactly $3$ (see Section 5 in \cite{LTWZ18}). We shall show below that $4$-edge-connected graphs with fewer vertices have flow index less than $3$.

\begin{lemma}\label{n<=5}
For a $4$-edge-connected graph $G$ on $n\le 5$ vertices, $\phi(G)<3$.
\end{lemma}
\begin{proof}
When $n\leq 2$, it holds by (i) of Lemma \ref{H}.
Suppose that $G$ is a minimal counterexample of the lemma with the least vertices.  Then $|V(G)|\geq 3$ and $G$ has no strongly-connected modulo $3$-orientation. If $G$ has an even degree vertex, by Mader's splitting lemma (see \cite{Mader78}), we can get a smaller counterexample. So the degree of each vertex of $G$ must be odd and $|V(G)|$ can only be $4$. By Lemma \ref{3edge}, $G$ does not contain three parallel edges, and so each vertex $v$ of $G$ has exactly $3$ neighbors. Thus $G$ can only be isomorphic to the graph $K_4^*$, and then $G\in\setCZ$ by Lemma \ref{H}, which is a contradiction.
\end{proof}

\section{Proofs of the main results}

Now we are ready to present the proofs of Theorems \ref{thmmain1} and \ref{main2m}. In fact, we shall prove a stronger version of Theorem \ref{main2m} with complete characterization of the bad attachment, stated as Theorem \ref{main2s}.
In this section, {\em we always let  $G$ be a simple graph with $\min\{\delta(G),\delta(G^c)\}\ge 4$}, where $G^c$ denotes the complement of $G$. For a vertex set $S\subset V(G)$, denote $\bar{S}=V(G)\setminus S$.

\begin{lemma}\label{4edgeconn}
If $G$ has an edge-cut of size at most $3$ and $|V(G)|\geq 26$, then $G^c\in \setCZ$.
\end{lemma}
\begin{proof}
Let $E_G(S,\bar{S})$ be an edge-cut of size at most $3$ in $G$. Since $\delta(G)\ge 4$, we have $$|S|(|S|-1)\ge 2|E(G[S])|\ge 4|S|-e_G(S,\bar{S})\ge 4|S|-3,$$
which implies $|S|\ge 5$. Similarly, we have $|\bar{S}|\ge 5$ as well. Since $\frac{1}{2}|V(G)|\geq 13$, one of $S$ and $\bar{S}$ has a size at least $13$, say $|\bar{S}|\geq 13$.

In $G^c$, consider the subgraph $E_{G^c}(S,\bar{S})$. It is almost a complete bipartite graph with at most $3$ edges deleted. Let $K_{s,t}$ be a maximal complete bipartite subgraph of $E_{G^c}(S,\bar{S})$ with $s=|S|\ge 5$. Then $t\ge |\bar{S}|-3\geq 10$. By Lemma \ref{bi}, $K_{s,t}\in \setCZ$. Let $S_1=\{x\in \bar{S}|e_{G^c}(x, S)\le 3\}$. Since $|S|\ge 5$ and $3\ge e_{G}(\bar{S}, S)\ge e_{G}(S_1, S)\ge|S_1||S|-3|S_1|$, we have $|S_1|\le 1$. This implies that $G^c$ is $4$-edge-connected since the only possible vertex in $S_1$ has at least $4$ edges connecting $\bar{S_1}$. Moreover, we have that $G^c=cl_3(K_{s,t})$. Thus $G^c\in \setCZ$ by Lemma \ref{3closure}.
\end{proof}

Define
$${\mathcal Y}_1=\{Y\subseteq V(G)|\ \exists H\subseteq G \text{ with } H\in \setCZ \text{ and } G[Y]=cl_3(H)\text{ in } G\}~{\text{and}}$$
$${\mathcal Y}_2=\{Y\subseteq V(G)|\ \exists H\subseteq G^c \text{ with } H\in \setCZ \text{ and } G^c[Y]=cl_3(H) \text{ in } G^c\}.~~$$

\begin{equation}\label{max3closure}
 \text{Choose}~ Y\in {\mathcal Y}_1\cup {\mathcal Y}_2 \text{  with } |Y| \text{ maximized }.
\end{equation}

\begin{lemma}\label{Ylarge4}
If $|V(G)|\geq 32$, then $|Y|\ge |V(G)|-4$.
\end{lemma}
\begin{proof}
If $|V(G)|\geq 32$, then one of $G, G^c$ has at least $\frac{1}{4}|V(G)|(|V(G)|-1)\ge 8(|V(G)|-1)$ edges. By Lemma \ref{czsize16}, it contains a subgraph $H\in\setCZ$ with $|V(H)|\ge 16$. Hence $|Y|\geq 16 $ by (\ref{max3closure}). Without loss of generality, assume that $Y\in {\mathcal Y}_1$.

Suppose, to the contrary, that $|\bar{Y}|\geq 5$. Since $G[Y]$ is a $3$-closure of a $\setCZ$-graph in $G$, we have
\begin{equation}\label{xbarY2}
\mbox{$e_G(Y,x)\leq2$ for each vertex $x\in \bar{Y}$.}
\end{equation}

We first show the following thing:
\begin{equation}\label{barYnotin}
\mbox{for any $Y_0\in{\mathcal Y_2}$, we have $\bar{Y}\not\subset Y_0$.}
\end{equation}
If $\bar{Y}\subset Y_0$, then $\bar{Y_0}\subset Y$.
For each $y\in \bar{Y_0}$, we have $e_{G^c}(y, Y_0)\le 2$, and so $e_{G^c}(y,\bar{Y})\le 2$, which gives $e_{G}(y,\bar{Y})\ge |\bar{Y}|-2$. Hence, together with (\ref{xbarY2}), we have $$2|\bar{Y}|\ge e_{G}(Y,\bar{Y})\ge (|\bar{Y}|-2)|\bar{Y_0}|,$$
which implies that $|\bar{Y_0}|\le 2|\bar{Y}|/(|\bar{Y}|-2)<4$ since $|\bar{Y}|\ge 5$ by the assumption. Hence $|Y_0|>|Y|+1$, and it contradicts the maximality of $|Y|$ in (\ref{max3closure}). This proves (\ref{barYnotin}).\\

Then we show the following thing:
\begin{equation}\label{barYlarge}
|\bar{Y}|\ge 15.
\end{equation}
In fact, if $|\bar{Y}|< 15$, then $|Y|\ge 18$ as $|V(G)|\ge 32$. Let $Z$ be a subset of $\bar{Y}$ with $|Z|=4$. Denote $Y'=\{y\in Y | e_G(y, Z)=0\}$. By (\ref{xbarY2}), there are at most $8$ vertices in $Y$ that are adjacent to some vertices in $Z$. So $|Y'|\ge |Y|-8\ge 10$. This implies that $E_{G^c}(Y', Z)$ forms a complete bipartite graph $H_1\cong K_{|Y'|, 4}\in\setCZ$ by Lemma \ref{bi}.

Now in $G^c$, consider the $3$-closure of $H_1$, namely $cl_3(H_1)$. We denote $Y_1=V(cl_3(H_1))$ in $G^c$ for convenience. By (\ref{xbarY2}), for each vertex $x\in \bar{Y}$, we have $e_{G^c}(Y',x)=|Y'|-e_{G}(Y',x)\ge10-2>3$, and so $x\in Y_1$ by definition. Thus $\bar{Y}\subset Y_1$. As $Y_1=V(cl_3(H_1))\in {\mathcal Y_2}$, it contradicts (\ref{barYnotin}), and hence this proves (\ref{barYlarge}).\\

Denote  $X=\{y\in Y | e_G(y, \bar{Y})\le 1\}$. If $|X|\ge 5$, we let $X_1$ be a subset of $X$ with $|X_1|=5$. Let $Z_1=\{z\in \bar{Y} | e_G(X, z)=0\}$. Then in $G$ there are at most $5$ vertices in $\bar{Y}$ that are adjacent to some vertices in $X_1$. So $|Z_1|\ge |\bar{Y}|-5\ge 10$ by (\ref{barYlarge}). Thus $E_{G^c}(X_1, Z_1)$ forms a complete bipartite graph $H_2\cong K_{5, |Z_1|}\in\setCZ$ by Lemma \ref{bi}. Now consider the $3$-closure of $H_2$ in $G^c$. Denote $Y_2=V(cl_3(H_2))$.  By (\ref{xbarY2}), for each vertex $z\in \bar{Y}$, we have $e_{G^c}(X_1,z)=|X_1|-e_{G}(X_1,z)\ge5-2=3,$ and so $z\in Y_2$ by definition.
This shows $\bar{Y}\subset Y_2$, a contradiction to (\ref{barYnotin}). Thus we must have $|X|\le 4$.

Since $|X|\le 4$ and $|Y|\ge 16$, we let $y_1,y_2\in Y\setminus X$ be two distinct vertices, that is, $e_G(y_i, \bar{Y})\ge 2$ for each $i=1,2$. Denote by $u_{i}, v_{i}\in \bar{Y}$ the two distinct neighbors of $y_i$ for each $i=1,2$. Let $Z$ be a subset of $\bar{Y}$ with $|Z|=4$ that contains $\{u_1,v_1\}\cup\{u_2,v_2\}$. Denote $Y'=\{y\in Y | e_G(y, Z)=0\}$. Then $y_i\in Y\setminus Y'$ with  $e_G(y_i, Z)\ge 2$ for $i=1,2$. By (\ref{xbarY2}), we have
$$2|Z|\ge e_G(Y\setminus Y', Z)=\sum_{i=1}^2 e_G(y_i, Z)+e_G((Y\setminus Y')\setminus\{y_1,y_2\}, Z)\ge 4+(|Y\setminus Y'|-2),$$
which implies that $|Y\setminus Y'|\le 2|Z|-2=6$, and so $|Y'|\ge |Y|-6\ge 10$.

Since $|Y'|\ge 10,$  we have that $E_{G^c}(Y', Z)$ forms a complete bipartite graph $H_3\cong K_{|Y'|, 4}\in\setCZ$ in $G^c$  by Lemma \ref{bi}.
Consider the $3$-closure of $H_3$ in $G^c$, and let $Y_3=V(cl_3(H_3))$. By (\ref{xbarY2}), for each vertex $x\in \bar{Y}$, we have $e_{G^c}(Y',x)=|Y'|-e_{G}(Y',x)\ge10-2>3,$ and hence $x\in Y_3$ by definition. Thus we have $\bar{Y}\subset Y_3$, which contradicts (\ref{barYnotin}). This completes the proof of Lemma \ref{Ylarge4}.
\end{proof}

\vspace{0.5cm}

\noindent{\bf Proof of Theorem \ref{thmmain1}:} By Lemma \ref{4edgeconn}, we may assume that both $G$ and $G^c$ are $4$-edge-connected. As in (\ref{max3closure}), we may, without loss of generality, assume that $Y\in {\mathcal Y}_1$. Thus $G[Y]=cl_3(H)$ for some subgraph $H\in\setCZ$ in $G$. Then $G/G[Y]$ has at most $5$ vertices by Lemma \ref{Ylarge4}. Since $G/G[Y]$ is $4$-edge-connected, we have $\phi(G/G[Y])<3$ by Lemma \ref{n<=5}, and so $\phi(G)<3$ by Lemma \ref{3closure}(i). This proves Theorem \ref{thmmain1}. $\blacksquare$

\vspace{1cm}

We shall prove the following theorem, which is stronger than Theorem \ref{main2m}. It provides a complete characterization of the bad attachment, and it also tells that the graph deleting the bad attachment is obtained from the $3$-closure of a $\setCZ$-graph.

\begin{figure}[!hpbt]
    \centering
    \includegraphics[width=0.8\textwidth,trim=65 525 55 75,clip]{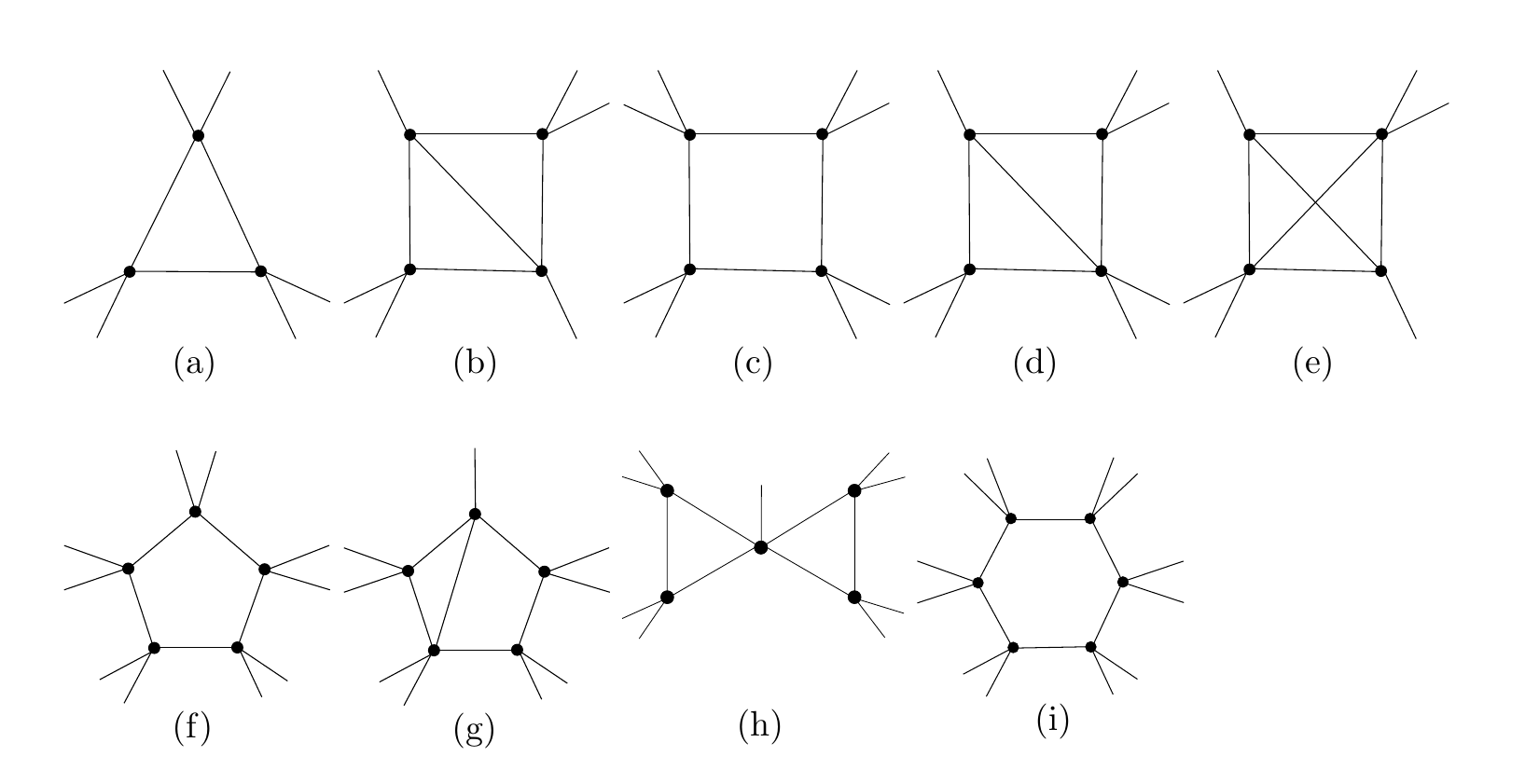}
    \caption{\small\it Characterization of all bad attachments.}\label{allbad}
\end{figure}

\begin{theorem}\label{main2s}
Let  $G$ be a simple graph with $|V(G)|\ge 73$. If $\min\{\delta(G),\delta(G^c)\}\ge 4$,
then one of the following statements holds:\\
(i)  $G\in \setCZ$ or $G^c\in\setCZ$.\\
(ii) both $G$ and $G^c$ are formed from the $3$-closure of a $\setCZ$-subgraph by adding a bad attachment isomorphic to Figure \ref{allbad} (c).\\
(iii) one of $G$ and $G^c$ is formed from the $3$-closure of a $\setCZ$-subgraph by adding a bad attachment isomorphic to Figure \ref{allbad} (a); the other is formed from the $3$-closure of a $\setCZ$-subgraph by adding a bad attachment isomorphic to Figure \ref{allbad} (a)-(i), or by adding two disjoint bad attachments isomorphic to Figure \ref{allbad} (a).
\end{theorem}

\vspace{0.5cm}

\noindent{\bf Proof of Theorem \ref{main2m} assuming Theorem \ref{main2s}:} By Remark 1, we know that if $G$ contains a bad attachment, then $G\notin\setCZ$.
Now it suffices to prove the ``moreover part'' of Theorem \ref{main2m}. Assume that both $G\notin\setCZ$ and $G^c\notin\setCZ$. Then both $G$ and $G^c$ are $4$-edge-connected by Lemma \ref{4edgeconn}. By Theorem \ref{main2s},   $G$ is formed from the $3$-closure of a subgraph $H\in\setCZ$ by adding a bad attachment or two. By the description of the bad attachment in Figure \ref{allbad} (a)-(i) in Theorem \ref{main2s}, $G/cl_3(H)$ is a $4$-edge-connected graph on at most $5$ vertices for Figure \ref{allbad} (a)-(e), or $G/cl_3(H)$ is an Eulerian graph (i.e. every vertex has an even degree) for Figure \ref{allbad} (f),(i) and for two disjoint bad attachments as Figure \ref{allbad} (a), or $G/cl_3(H)$ is a $4$-edge-connected graph with two odd vertices for Figure \ref{allbad} (g),(h). In each case, we have that $\phi(G/cl_3(H))<3$  by Lemma \ref{n<=5} or by constructing a strongly-connected modulo $3$-orientation. Thus $\phi(G)<3$ by Lemma \ref{3closure}(i).  The same proof works for $G^c$ to show $\phi(G^c)<3$. This finishes the proof of Theorem \ref{main2m}.  $\blacksquare$

\vspace{1cm}

Before proving Theorem \ref{main2s}, we will show that some more graphs are in $\setCZ$. Each of these graphs has only one more edge than the responding bad attachment, and any graph obtained from one of them by adding edges is in $\setCZ$ by Observation \ref{inandnotin}.

\begin{lemma}\label{L}
Each of the graphs in Figure \ref{edgeadded} is in $\setCZ$.
\end{lemma}
\begin{proof}
For each $1\le i\le 9$, let $G=L_i$ be a graph with $v,x,b\in V(G)$ as in Figure \ref{edgeadded}. Then it is easy to check that $G_{[v,xb]}$ is $4$-edge-connected and $G_{[v,xb]}=cl_3(x)$, and thus $G_{[v,xb]}\in \setCZ$ by Lemma \ref{3closure} (ii). It follows that $G\in \setCZ$ from Lemma \ref{split1}.
\end{proof}

\begin{figure}[!hpbt]
\centering
\includegraphics[width=0.8\textwidth,trim=60 515 60 75,clip]{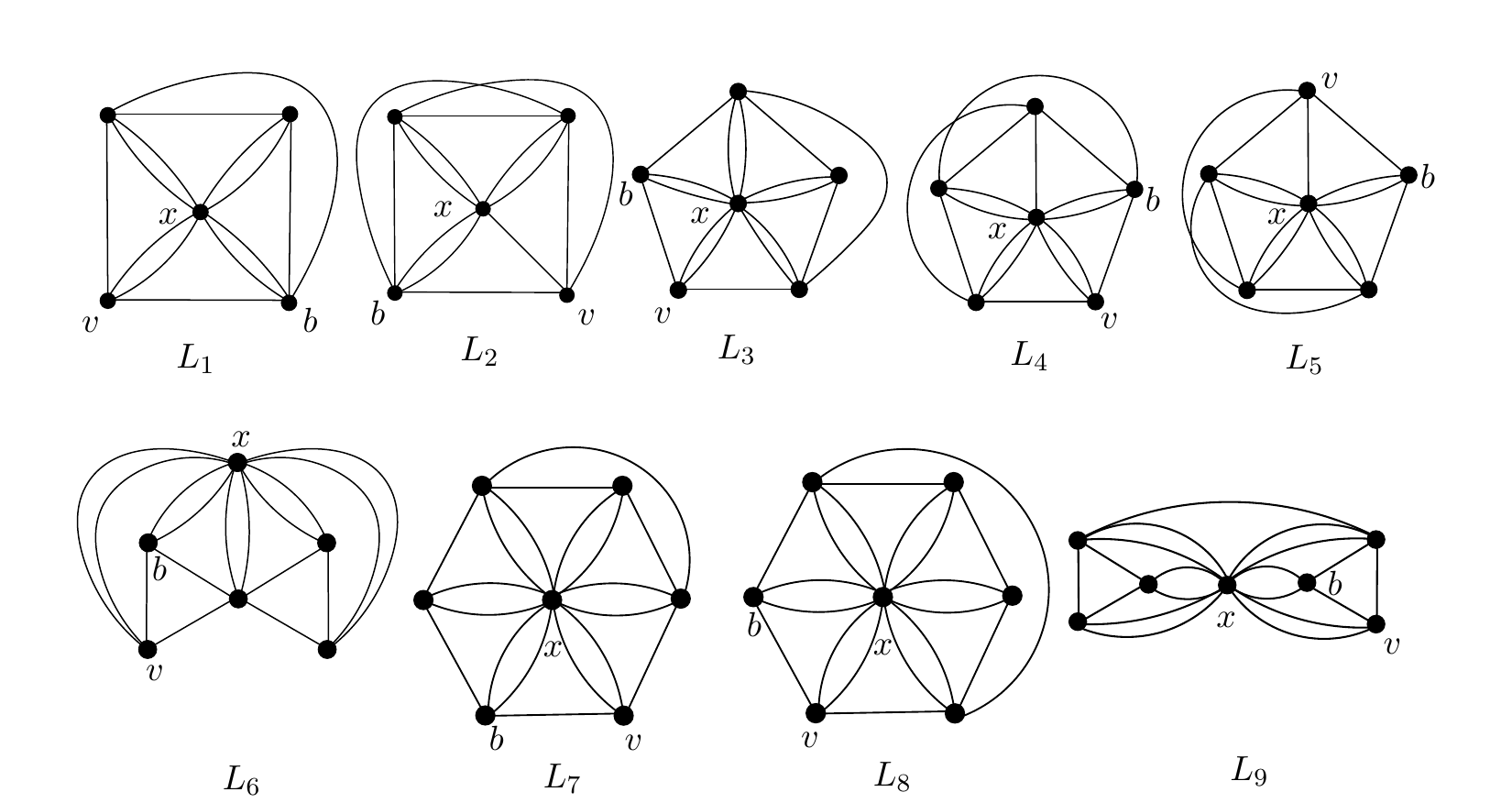}

\caption{\small\it A single edge added to each of the bad attachments.}\label{edgeadded}
\end{figure}

\noindent{\bf Proof of Theorem \ref{main2s}:} By Lemma \ref{4edgeconn}, we may assume that both $G$ and $G^c$ are $4$-edge-connected. As in (\ref{max3closure}), we choose  $Y\in {\mathcal Y}_1\cup {\mathcal Y}_2$ with $|Y|$ maximized. Without loss of generality, assume $Y\in {\mathcal Y}_1$. Let $X=\{x\in Y | e_G(x,\bar{Y})>0\}$.  Since $Y$ is a $3$-closure, for each vertex $x\in \bar{Y}$, $e_G(Y,x)\leq 2$. Thus
\begin{equation}\label{Xsize}
|X|\le e_G(X,\bar{Y})=e_G(Y,\bar{Y})\le 2|\bar{Y}|.
\end{equation}
Since $\delta(G)\ge 4$, we also have
\begin{equation}\label{barY34}
 4|\bar{Y}|-|\bar{Y}|(|\bar{Y}|-1)|\le e_G(Y,\bar{Y})\le 2|\bar{Y}|,
\end{equation}
which, together with Lemma \ref{Ylarge4}, shows that $3\le |\bar{Y}|\le 4$. We shall distinguish our discussion according to the value of $|\bar{Y}|$.\\

\noindent{\bf Case A} $|\bar{Y}|=3$.

By (\ref{barY34}), we have that $e_G(Y,\bar{Y})=6$ and $G[Y]$ forms a triangle. Thus this bad attachment of $G$ is isomorphic to Figure \ref{allbad} (a).
It follows from (\ref{Xsize}) that $|X|\le6.$ Since $|V(G)|\geq73$, we have $|Y\setminus X|\ge64$.

In the complementary graph $G^c$, $E_{G^c}(Y\setminus X,\bar{Y})$ forms a complete bipartite graph $K_{3, |Y\setminus X|}$. Consider the subgraph $G^c[Y\setminus X]$ induced by $Y\setminus X$ in $G^c$. Let $X_1$ be the set of non-isolated vertices in $G^c[Y\setminus X]$. If $|X_1|\ge 14$, then $G^c[X_1\cup \bar{Y}]$ forms a graph $H_1\cong K_{3,|X_1|}^+\in \setCZ$ by Lemma \ref{3t}. Otherwise, we have $|X_1|\le 13$, which implies that there are at least $51$ isolated vertices in $G^c[Y\setminus X]$. Since $\delta(G^c)\ge 4$ and $|\bar{Y}|=3$, each isolated vertex in $G^c[Y\setminus X]$ is connected to $X$. Since $|X|\le 6$ and $|(Y\setminus X) \setminus X_1|\ge 51$, there exists a vertex $x_0\in X$ such that $e_{G^c}(x_0, (Y\setminus X) \setminus X_1)\ge 10$ by Pigeon-Hole principle.
Let $X_0=\{y\in Y\setminus X | e_{G^c}(x_0, y)>0\}$. Then $|X_0|\ge 10$ and $E_{G^c}(X_0, \bar{Y}\cup\{x_0\})$ forms a complete bipartite graph $H_2=K_{4,|X_0|}\in \setCZ$ by Lemma \ref{bi}. Therefore, we can always find a $\setCZ$-subgraph $H\in\{H_1, H_2\}$ in $G^c$ that contains $\bar{Y}$.
Now consider the $3$-closure of $H$ in $G^c$ and let $Z=V(cl_3(H))$. Denote $s=|\bar{Z}|$. Since $E_{G^c}(Y\setminus X,\bar{Y})$ forms a complete bipartite graph $K_{3, |Y\setminus X|}$, we have $Y\setminus X\subset Z$, which is $\bar{Z}\subseteq X$. Then by (\ref{max3closure}), $$3=|\bar{Y}|\le s \le |X|\le 6.$$

For each vertex $x\in \bar{Z}$, $e_{G^c}(Z,x)\leq 2$, and thus $e_{G^c}(Z,\bar{Z})\leq 2s$.
Since $\min\{\delta(G)$, $\delta(G^c)$$\}$ $\geq4$ and $e_G(\bar{Z},\bar{Y})\leq e_G(Y,\bar{Y})\leq6$, we have $s(s-1)+e_{G^c}(Z,\bar{Z})\geq 4s$ and $e_{G^c}(Z,\bar{Z})\ge |\bar{Z}||\bar{Y}|-e_G(\bar{Z},\bar{Y})\ge 3s-6$. In summary,
\begin{equation}\label{attachedge}
\max\{3s-6,5s-s^2\}\leq e_{G^c}(Z,\bar{Z})\leq 2s.
\end{equation}
Since $3\le s\le 6$, we shall discuss the following cases, characterizing all the bad attachments in Theorem \ref{main2s} (iii).

\begin{itemize}
\item $s=3$.

By (\ref{attachedge}), we have $e_{G^c}(Z,\bar{Z})=6$. Then the only possibility is that $\bar{Z}$ induces a bad attachment isomorphic to Figure \ref{allbad} (a) in $G^c$.

\item $s=4$.

Then $6\le e_{G^c}(Z,\bar{Z})\le 8$ by (\ref{attachedge}).  If $e_{G^c}(Z,\bar{Z})=6$, then $\delta(G^c)\ge 4$ forces that the bad attachment induced by $\bar{Z}$ is isomorphic to Figure \ref{allbad} (b) or (e).

If $e_{G^c}(Z,\bar{Z})=7$, then $\delta(G^c)\ge 4$ forces that $G^c[\bar{Z}]$ has at least $5$ edges. If $G^c[\bar{Z}]\cong K_4$, then $G^c/cl_3(H)\cong L_1 \in\setCZ$ by Lemma \ref{L}, and so $G^c\in\setCZ$ by Lemma \ref{3closure}(ii). Hence,  Theorem \ref{main2s} (i) holds. Otherwise, $G^c[\bar{Z}]$ has exactly $5$ edges, and the bad attachment induced by $\bar{Z}$ is isomorphic to Figure \ref{allbad} (d).

If $e_{G^c}(Z,\bar{Z})=8$, then $\delta(G^c)\ge 4$ implies that $G^c[\bar{Z}]$ contains a cycle $C_4$. If $G^c[\bar{Z}]\cong C_4$, then the bad attachment induced by $\bar{Z}$ is isomorphic to Figure \ref{allbad} (c). Otherwise, $G^c[\bar{Z}]$ has at least $5$ edges, and $G^c/cl_3(H)$ contains a subgraph $ L_2 \in\setCZ$ by Lemma \ref{L}. This shows that $G^c\in\setCZ$ by Lemma \ref{3closure}(ii), and so Theorem \ref{main2s} (i) holds.

\item $s=5$.

By (\ref{attachedge}), we have $9\le e_{G^c}(Z,\bar{Z})\le 10$, and $\delta(G^c)\ge 4$ implies that $G^c[\bar{Z}]$ contains a cycle $C_5$ or a hourglass graph $T_2$ which consists of two triangles with a common vertex (see Figure \ref{MMT2}).

If $e_{G^c}(Z,\bar{Z})=10$, then the bad attachment induced by $\bar{Z}$ is isomorphic to Figure \ref{allbad} (f) when $G^c[\bar{Z}]\cong C_5$. Assume that $G^c[\bar{Z}]$ contains a cycle $C_5$ plus a chord. Then $G^c/cl_3(H)$ contains a subgraph $ L_3 \in\setCZ$ by Lemma \ref{L}. Therefore, $G^c\in\setCZ$ by Lemma \ref{3closure}(ii), and so Theorem \ref{main2s} (i) holds.

If $e_{G^c}(Z,\bar{Z})=9$, then $\delta(G^c)\ge 4$ further forces that $G^c[\bar{Z}]$ contains a cycle $C_5$ plus a chord or a $T_2$. When $G^c[\bar{Z}]$ contains additional edges, $G^c/cl_3(H)$ contains $L_4$, $L_5$ or $L_6 \in\setCZ$ , and so $G^c\in\setCZ$. Otherwise, the bad attachment induced by $\bar{Z}$ is isomorphic to Figure \ref{allbad} (g) or (h).

\begin{figure}[!hpbt]
\centering
\includegraphics[width=0.9\textwidth,trim=50 595 65 80,clip]{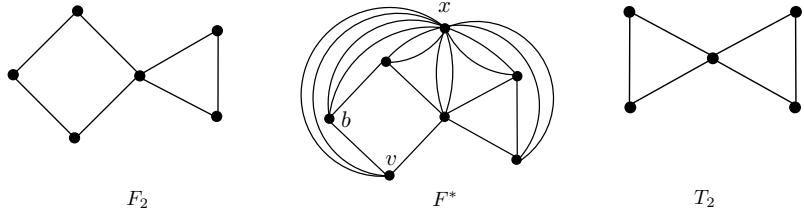}

\caption{\small\it The graphs $F$, $F^*$ and $T_2$.}\label{MMT2}
\end{figure}

\item $s=6$.

Then $e_{G^c}(Z,\bar{Z})=12$ by (\ref{attachedge}). Define the fish graph $F_2$ as a $4$-cycle attached to a triangle with a common vertex (see Figure \ref{MMT2}). Since $\delta(G^c)\ge 4$, $G^c[\bar{Z}]$ has minimal degree at least $2$, we deduce that $G^c[\bar{Z}]$ contains a $C_6$, an $F_2$ or two disjoint triangles.

When $G^c[\bar{Z}]$ contains an $F_2$, $G^c/cl_3(H)$ contains a graph $F^*$ as in Figure \ref{MMT2}. Since $F^*_{[v,xb]}=cl_3(x)$ and it is $4$-edge-connected, we have $F^*_{[v,xb]}\in\setCZ$ by Lemma \ref{3closure}. Then $F^*\in \setCZ$ by Lemma \ref {split1}, and so $G^c\in \setCZ$ by Lemma \ref{3closure}.

If $G^c[\bar{Z}]$ contains a cycle $C_6$ plus a chord, then $G^c/cl_3(H)$ contains $L_7$ or $L_8 \in\setCZ$, and so $G^c\in\setCZ$.  If $G^c[\bar{Z}]$ contains two disjoint triangles plus an additional edge, then $G^c/cl_3(H)$ contains $L_9 \in\setCZ$. Thus $G^c\in\setCZ$ and Theorem \ref{main2s} (i) holds.  Otherwise, the bad attachment induced by $\bar{Z}$ is isomorphic to Figure \ref{allbad} (i), or two disjoint bad attachments isomorphic to Figure \ref{allbad} (a).
\end{itemize}

\vspace{0.5cm}
\noindent{\bf Case B} $|\bar{Y}|=4$.

By (\ref{Xsize}), we have $|X|\le 8$, and so $|Y\setminus X|= |Y|-|X|\ge 61>10$. Then in $G^c$, $E_{G^c}(Y\setminus X,\bar{Y})$ forms a complete bipartite graph $H\cong K_{4, |Y\setminus X|}\in\setCZ$ by Lemma \ref{bi}. Consider the $3$-closure of $H$ in $G^c$ and let $Z=V(cl_3(H))$. Then $\bar{Y}\subset Z$ and $\bar{Z}\subseteq X$. For each vertex $x\in \bar{Z}$, we have $e_{G^c}(x,\bar{Y})\leq e_{G^c}(x,Z)\leq 2$ by definition,  and so  $$e_{G^c}(\bar{Z},\bar{Y})\leq 2|Z|.$$
On the other hand, we have $e_{G}(\bar{Z}, \bar{Y})\leq e_{G}(X,\bar{Y})\leq 2|\bar{Y}|=8$ by (\ref{Xsize}), and hence
$$e_{G^c}(\bar{Z},\bar{Y})= |\bar{Z}||\bar{Y}|-e_{G}(\bar{Z},\bar{Y})\ge 4|Z|-8.$$
Thus $4|\bar{Z}|-8\le 2|Z|$, i.e., $|Z|\leq 4$. By the maximality of $Y$ in (\ref{max3closure}), we must have $|Z|=4$. Therefore, all the inequalities above are exactly equalities. Thus we have $e_{G}(Y,\bar{Y})=e_{G}(\bar{Z},\bar{Y})=8$ and $e_{G^c}(Z,\bar{Z})=e_{G^c}(\bar{Y},\bar{Z})=8$.

Now we will adapt the same argument as in the proof of $s=4$ in Case A. Notice that $G[\bar{Y}]$ contains a cycle $C_4$ since $\delta(G)\ge 4$. If  $G[\bar{Y}]$ has at least $5$ edges, then $G/G[Y]$ contains a subgraph $L_2 \in\setCZ$ by Lemma \ref{L}. This shows that $G\in\setCZ$ by Lemma \ref{3closure}(ii), and so Theorem \ref{main2s} (i) holds. Otherwise, $G[\bar{Y}]$ is exactly a cycle $C_4$. Then in $G$ the bad attachment induced by $\bar{Y}$ is isomorphic to Figure \ref{allbad} (c). Analogously, either $G^c/cl_3(H)\in \setCZ$ or the bad attachment of $G^c$ induced by $\bar{Z}$ is isomorphic to Figure \ref{allbad} (c). This completes the proof of Theorem \ref{main2s}. $\blacksquare$

\section*{Acknowledgments}
Jiaao Li is partially supported by the Fundamental Research Funds for the Central Universities.
Xueliang Li and Meilin Wang are  partially supported by NSFC No.11871034, 11531011 and NSFQH No.2017-ZJ-790.


{\small
\begin{thebibliography}{99}

\bibitem{GTZ98}
L. A. Goddyn, M. Tarsi and C.-Q. Zhang, On $(k,d)$-colorings and
fractional nowhere-zero flows, J. Graph Theory 28 (1998) 155-161.

\bibitem{HLLZ12}
X. Hou, H.-J. Lai, P. Li and C.-Q. Zhang, Group connectivity of complementary graphs, J. Graph Theory 69 (2012) 464-470.

\bibitem{Jaeger79}
F. Jaeger, Flows and generalized coloring theorems in graphs, J. Combin. Theory Ser. B 26 (1979) 205-216.

\bibitem{Jaeger88}
F. Jaeger, Nowhere-zero flow problems, in: {\it Selected Topics in Graph Theory}, vol.3, L. Beineke
and R. Wilson, eds., Academic Press, London/New York, 1988, pp. 91-95.

\bibitem{LTWZ18}
J. Li, C. Thomassen, Y. Wu and C.-Q. Zhang, The flow index and strongly connected orientations,
European J. Combin. 70 (2018) 164-177.

\bibitem{LTWZ13} L. M. Lov\'{a}sz, C. Thomassen, Y. Wu and C.-Q. Zhang, Nowhere-zero $3$-flows and modulo $k$-orientations,
J. Combin. Theory Ser. B 103 (2013) 587-598.

\bibitem{Mader78}
W. Mader, A reduction method for edge-connectivity in graphs, Ann. Discrete Math. 3 (1978)
 145-164.

\bibitem{Nash64}
C. St. J. A. Nash-Williams, Decomposition of finite graphs into forest, J.
London Math. Soc. 39 (1964) 12.

\bibitem{Thom12}
C. Thomassen, The weak $3$-flow conjecture and the weak circular flow conjecture, J. Combin.
Theory Ser. B 102 (2012) 521-529.

\bibitem{YLL10}
X. J. Yao, X. Li and H.-J. Lai, Degree conditions for group connectivity, Discrete Math. 310 (2010) 1050-1058.

\bibitem{Younger83}
D. H. Younger, Integer flows, J. Graph Theory 7 (1983) 349-357.

\bibitem{Zhang97}
C.-Q. Zhang, {\it Integer Flows and Cycle Covers of Graphs,} Marcel Dekker Inc. New York, 1997.

\end{thebibliography}
}
\end{document}